\documentclass[12pt]{article}
\usepackage[centertags]{amsmath}
\usepackage{amsfonts}
\usepackage{amssymb,amsmath,amsfonts}
\usepackage{booktabs}
\usepackage{array, color}
\usepackage{placeins}
\usepackage{float}
\usepackage{relsize}
\usepackage{lipsum}
\usepackage{graphicx}
\usepackage{subcaption}
\usepackage{mathrsfs} 
\usepackage{float}
\restylefloat{table}
\usepackage{cite}
\usepackage{derivative}
\usepackage{multirow}
\usepackage{undertilde}
\usepackage{amssymb,amsmath,amsfonts,fancyhdr}
\usepackage{amsmath}
\usepackage{bigints}
\numberwithin{equation}{section}
\usepackage{amsthm}
\usepackage{mathtools}

\usepackage[utf8]{inputenc}
\usepackage[english]{babel}
\newtheorem{theorem}{Theorem}[section]

\usepackage{bm}
\usepackage{amssymb}

\usepackage{booktabs}
\usepackage{color}
\usepackage{subfloat}
\usepackage{hyperref}
\usepackage{multirow}

\usepackage{lscape}
\usepackage{natbib}
\makeatletter
\newcommand{\thickhline}{%
    \noalign {\ifnum 0=`}\fi \hrule height 1pt
    \futurelet \reserved@a \@xhline
}
\newcolumntype{"}{@{\hskip\tabcolsep\vrule width 1pt\hskip\tabcolsep}}
\makeatother
\makeatletter \oddsidemargin  -.1in \evensidemargin -.1in
\newcommand{\doublespacing}{\let\CS=\@currsize
 \renewcommand{\baselinestretch}{1.05}\tiny\CS}
\begin{document}
\newcommand{\bea}{\begin{eqnarray}}
\newcommand{\eea}{\end{eqnarray}}
\newcommand{\nn}{\nonumber}
\newcommand{\bee}{\begin{eqnarray*}}
\newcommand{\eee}{\end{eqnarray*}}
\newcommand{\lb}{\label}
\newcommand{\nii}{\noindent}
\newcommand{\ii}{\indent}
\newtheorem{thm}{Theorem}[section]
\newtheorem{lem}{Lemma}[section]
\newtheorem{rem}{Remark}[section]
\theoremstyle{remark}
\renewcommand{\theequation}{\thesection.\arabic{equation}}
\date{}
\vspace{5cm}
\title{\bf On Estimating the Selected Treatment Mean under a Two-Stage Adaptive Design}
\author{Masihuddin$^{1}$\thanks{Email : masih.iitk@gmail.com, masihst@iitk.ac.in}
	~~	
	and~	Neeraj Misra$^{2}$\thanks{Email : neeraj@iitk.ac.in}}
\maketitle \noindent {\it $^{1,2}$ Department of Mathematics \& Statistics, Indian
Institute of Technology Kanpur, Kanpur-208016, Uttar Pradesh, India} 

\newcommand{\oddhead}{ Estimation of the Selected Treatment Mean}
\renewcommand{\@oddhead}
{\hspace*{-3pt}\raisebox{-3pt}[\headheight][0pt]
{\vbox{\hbox to \textwidth
{\hfill\oddhead}\vskip8pt}}}
\vspace*{0.05in}
\noindent {\bf Abstract}: 
Adaptive designs are commonly used in clinical and drug development studies for optimum utilization of available resources. In this article, we consider the problem of estimating the effect of the selected (better) treatment using a two-stage adaptive design. Consider two treatments with their effectiveness characterized by two normal distributions having different unknown means and a common unknown variance. The treatment associated with the larger mean effect is labeled as the better treatment. In the first stage of the design, each of the two treatments is independently administered to different sets of $n_1$ subjects, and the treatment with the larger sample mean is chosen as the better treatment. In the second stage, the selected treatment is further administered to $n_2$ additional subjects. In this article, we deal with the problem of estimating the mean of the selected treatment using the above adaptive design. We extend the result of \cite{cohen1989two} by obtaining the uniformly minimum variance conditionally unbiased estimator (UMVCUE) of the mean effect of the selected treatment when multiple observations are available in the second stage. We show that the maximum likelihood estimator (a weighted sample average based on the first and the second stage data) is minimax and admissible for estimating the mean effect of the selected treatment.  We also propose some plug-in estimators obtained by plugging in the pooled sample variance in place of the common variance $\sigma^2$, in some of the estimators proposed by \cite{misra2022estimation} for the situations where $\sigma^2$ is known. The performances of various estimators of the mean effect of the selected treatment are compared via a simulation study. For the illustration purpose, we also provide a real-data application. \vspace{2mm}\\
\noindent {\it AMS 2010 SUBJECT CLASSIFICATIONS:} 62F07 · 62F10 · 62C20\\

\newcommand\blfootnote[1]{%
	\begingroup
	\renewcommand\thefootnote{}\footnote{#1}%
	\addtocounter{footnote}{-1}%
	\endgroup
}\blfootnote{
}
\section{Introduction}
Adaptive designs have a growing importance in clinical drug discovery and development. In clinical studies, multiple new treatments are often of interest for evaluation but, due to limited resources (time, available patients, budget, etc.), only one or two with the best-observed response(s) can be selected for further assessment in a large-scale clinical trial. Two-stage adaptive designs mainly deal with finding a safe and effective treatment among multiple candidate treatments in stage 1 and then validating its properties using an independent sample in stage 2. \par 
For a thorough insightful overview on adaptive designs in clinical trials, the reader is referred to \cite{pallmann2018adaptive}. For an extensive discussion on inference procedures in two-stage adaptive designs, one may refer to \cite{bauer1999combining}, \cite{sampson2005drop}, \cite{stallard2008group}, \cite{bowden2008unbiased}, \cite{carreras2013shrinkage}, \cite{kimani2013conditionally}, \cite{chiu2018design}, \cite{kimani2020point} and \cite{robertson2022point}. \par  
In this paper, we consider a two-stage adaptive design comprising two stages with selection of a candidate for the better treatment in the first stage and estimation of the mean treatment effect of the selected treatment in the second stage.
It is well known that, using a single-stage data alone, there does not exist any unbiased estimator of the selected mean in many cases. For example, the selected means of normal and binomial populations are not unbiasedly estimable (see \cite{putter1968estimating} and \cite{tappin1992unbiased}). However, the naive estimates that incorporate data from both the stages can induce selection bias. To overcome this issue, the technique of Rao-Blackwellization can be utilized. Using this technique, the unbiased second stage sample mean is conditioned on a complete-sufficient statistic. As a result, a uniformly minimum variance conditionally unbiased estimator (UMVCUE) is obtained. An appealing property of the UMVCUE is that it has the smallest variance (or, in other words, the smallest mean squared error  (MSE)) among all the conditionally unbiased estimators of the selected mean. \par 

Initially, \cite{cohen1989two} dealt with the two-stage estimation of the selected treatment mean under the ranking and selection framework in the Gaussian setting. They separately considered the cases of known and unknown variance. The authors obtained the UMVCUE for the selected normal mean which uses data from both the stages. A limitation of \cite{cohen1989two} work is that it assumes single observation in the second stage of the adaptive design. Since then, their work has been extended by many researchers including \cite{tappin1992unbiased} and \cite{sampson2005drop}. For the case of common known variance, \cite{bowden2008unbiased} extended the UMVCUE to account for unequal stage one and stage two sample sizes, where the parameter of interest is the $j^{th}$ best among the $k$ candidates. \cite{kimani2013conditionally} considered point estimation of the selected most effective treatment compared with a control, after a two-stage adaptive seamless trial in which treatment selection and the possibility of early stopping for futility are available at stage 1. Using a multistage analog of the two-stage drop-the-losers design, \cite{bowden2014conditionally} provided unbiased and near unbiased estimates for the selected mean. \par
 In many practical situations, it may not be appropriate to assume the variances of the treatment effects to be known. For the case of common unknown variance, building upon the work of \cite{cohen1989two}, we derive the UMVCUE for the selected treatment mean when there are more than one observations in the second stage of the adaptive design. Although, \cite{robertson2019conditionally} had claimed that their UMVCUE works for multiple observations in the second stage, their approach is not optimal as they have conditioned the second stage data on a statistic which is not a complete-sufficient statistic. Our extended UMVCUE takes care of this deficiency. We have also obtained a minimax estimator for the selected treatment mean under the scaled mean squared error criterion. This minimax estimator is also shown to be admissible. \par 
The remainder of the paper is organized as follows: In Section 2, we introduce some notations and preliminaries that will be used all across the paper. In Subsection 2.1, we derive the UMVCUE of the selected treatment mean. In Section 3, we prove that the naive estimator, which is the weighted average of the first and the second stage sample means, is  minimax and admissible for estimating the selected treatment mean in case of common unknown variance. In Subsection 3.1, we provide some additional estimators for the selected treatment mean. In order to have  numerical assessment of the performances of various competing estimators under the criterion of the scaled mean squared error and the scaled 
bias, we provide a simulation study in Section 4. For illustration, a real data set has also been considered  in Section 5 of the paper.
  
\section{Estimation of the selected treatment mean}
The notations listed below will be used throughout the article:\\
\begin{itemize}
	\item[ $\bullet$] $\mathbb{R}$: the real line $\left(-\infty,\infty\right)$;
	\item[ $\bullet$] $\mathbb{R}^k$: the $k$ dimensional Euclidean space, $k \in \{2,3,\ldots\}$;
	\item[ $\bullet$] $N\left(\theta, \sigma^2 \right)$: normal distribution with mean $\theta\in \mathbb{R}$ and standard deviation $\sigma \in \left(0,\infty\right) $;
	\item[ $\bullet$] $\phi(\cdot) $: probability density function (p.d.f.) of $N(0,1)$;
	\item[ $\bullet$] $\Phi(\cdot)$: cumulative distribution function (c.d.f.) of $N(0,1)$;
	\item[ $\bullet$] $Beta (a,b)$: beta distribution with shape parameters $a>0$ and $b>0$;
	\item[ $\bullet$] $B(\alpha,\beta)=\frac{\Gamma(\alpha) \Gamma(\beta)}{\Gamma(\alpha+\beta)}$, $\alpha >0$, $\beta>0$, will denote the usual beta function and $\Gamma (\cdot )$ will denote the usual gamma function;
	\item[ $\bullet$] For real numbers  $x$ and $y$
	\begin{equation*}
		\label{S1.E1}
		I(x\geq y)=\begin{cases}
			1, & \text{if $x \geq y$ }\\
			0, & \text{if $x < y$}.
		\end{cases}
	\end{equation*}
 $I(x > y)$ is defined similarly.
\end{itemize}
Consider two treatments say, $\tau_1$ and $\tau_2$, such that effectiveness of the treatment $\tau_i$ is described by $N$$\left(\mu_i, \sigma^2 \right),~ i=1,2$; where $\mu_1 \in\mathbb{R}$ and $\mu_2 \in \mathbb{R}$ are unknown mean treatment effects and $\sigma^2 ~(\sigma >0)$ is the common unknown treatment effect variance. We define the treatment associated with $\max\{\mu_{1},\mu_{2}\}$ as the better or the promising treatment. We consider an adaptive design which consists of two stages, with a single data-driven selection made in the interim. In the first stage of the design, say stage $1$, the treatment $\tau_1$ is administered to $n_1$ respondents and the treatment $\tau_2$ is independently administered to another set of $n_1$ respondents. Let $\overline{X}_i,~ i=1,2,$ be  the sample averages (mean treatment effect estimates) corresponding to the two treatments. For the purpose of selecting the better treatment, we consider the natural selection rule that selects the treatment with the larger sample mean as the better treatment (for optimality properties of this natural selection rule, see \cite{bahadur1952impartial}, \cite{eaton1967some} and \cite{misra1994non}). Let $Q \in \{1,2\}$ be the index of the selected treatment $\tau_Q$ (i.e. $Q=1$, if $\overline{X}_1 > \overline{X}_2$,  $Q=2$, if $\overline{X}_2 \geq \overline{X}_1$). Treatment $\tau_Q$ is then carried forward to the second stage, referred to as the stage $2$ of the two stage design, for further analysis. In stage 2, the selected treatment $\tau_Q$ is independently administered to $n_2$ additional respondents. Let $\overline{Y}$ be the stage $2$ sample mean for the selected treatment $\tau_Q$. The goal is to estimate $\mu_{Q}$, the mean effect of the selected treatment. Note that $\mu_{Q} \equiv \mu_{Q} \left(\mu_{1},\mu_{2},\overline{X}_1,\overline{X}_2\right)$ is a random parameter which depends on $\mu_{1}$, $\mu_{2}$, $\overline{X}_1$ and $\overline{X}_2$, equals $\mu_{1}$, if $\overline{X}_1  > \overline{X}_2$, and equals $\mu_{2}$, if $\overline{X}_2  \geq \overline{X}_1$. Clearly, $\overline{X}_i \sim N\left(\mu_i, \frac{\sigma^2}{n_1} \right),~i=1,2,$ are independently distributed  and, conditioned on $Q$, $\overline{Y} \sim N\left(\mu_Q, \frac{\sigma^2}{n_2} \right)$. \par 

\par
\noindent The following notations will also be utilized throughout the paper: \\
\noindent $\utilde{T} = (\overline{X}_1,\overline{X}_2, \overline{Y},S^2)$; $S^2=\sum\limits_{i=1}^{2} \sum\limits_{j=1}^{n_1} X_{ij}^2+\sum\limits_{j=1}^{n_2} Y_{j}^2$; $\underline{\theta} = (\mu_1,\mu_2,\sigma)$; $\Theta = \mathbb{R}^2 \times (0,\infty)$;  $\overline{X}_{Q} = \max (\overline{X}_1 ,\overline{X}_2)$ (maximum of $\overline{X}_1$ and $\overline{X}_2$); $\overline{X}_{3-Q} = \min (\overline{X}_1 ,\overline{X}_2)$ (minimum of $\overline{X}_1$ and $\overline{X}_2$).  Also, for any $\underline{\theta} \in \Theta$, $\mathbb{P}_{\underline{\theta}}(\cdot)$ will denote the probability  measure induced by $\utilde{T} = (\overline{X}_1,\overline{X}_2, \overline{Y},S^2)$, when $\underline{\theta} \in \Theta$ is the true parameter value, and $\mathbb{E}_{\underline{\theta}}(\cdot)$ will denote the expectation operator under the probability measure $\mathbb{P}_{\theta}(\cdot)$, $\underline{\theta} \in \Theta$. \par
\noindent In this paper, our focus is on point estimation of the selected treatment mean effect defined by
\begin{align}
	\label{S1.E1}
	\mu_Q\equiv \mu_{Q}\left(\mu_{1},\mu_{2}, \overline{X}_1,\overline{X}_2\right)&=\begin{cases}
		\mu_1, & \text{if $\overline{X}_1  \geq \overline{X}_2$ }\\
		\mu_2, & \text{if $\overline{X}_1 < \overline{X}_2$}
	\end{cases},\nonumber\\
	&=\mu_{1}I\left(\overline{X}_1 \geq \overline{X}_2\right) +\mu_{2} I\left(\overline{X}_2 >\overline{X}_1\right),
\end{align}
under the scaled squared error loss function
\begin{equation}	\label{S2.E2}
	L_{\utilde{T}}(\underline{\theta},a) = \left(\frac{a-\mu_Q}{\sigma} \right)^2  ,~~ \underline{\theta} \in \Theta, ~ a \in \mathcal{A} = \mathbb{R}.
\end{equation}
It is worth mentioning here that the statistic $\left(\frac{n_1\overline{X}_Q+n_2\overline{Y}}{n_1+n_2},\overline{X}_{3-Q},S^2,Q\right)$ is minimal sufficient but not complete. However, given $Q$, the statistic $\left(\frac{n_1\overline{X}_Q+n_2\overline{Y}}{n_1+n_2},\overline{X}_{3-Q},S^2\right)$ is a complete-sufficient statistic. Consequently, $\utilde{T} = (\overline{X}_1,\overline{X}_2, \overline{Y},S^2)$ is a sufficient statistic and $\mu_{Q}$ depends on observations only through $\utilde{T}$. Therefore, we may restrict our attention to only those estimators that depend on observations only through $\utilde{T}$ (see \cite{misra1993umvue}) . \par 
Under the scaled squared error loss function $(2.2)$, the risk function (also referred to as the scaled mean squared error) of an estimator $d(\utilde{T})$ is defined by
$$ R(\underline{\theta}, d)=\mathbb{E}_{\underline{\theta}}\left(\frac{d(\utilde{T})-\mu_Q}{\sigma}\right)^2,~ \underline{\theta} \in \Theta.$$ 
Suppose, we have a prior distribution (density) $\Pi$ on $\Theta$. Then, the Bayes risk of an estimator $d$,  with respect to the prior, $\Pi$ is defined as 
\begin{align*}
	r(\Pi, d)&=\mathbb{E}_{\Pi} \left(  R(\underline{\theta}, d) \right)\\
	&=\displaystyle \int_{\Theta} \displaystyle \int_{\zeta} L_{\utilde{T}}(\underline{\theta},d(\utilde{t})) f(\utilde{t}|\underline{\theta}) \Pi(\underline{\theta})d \underline{t}~ d \underline{\theta},
\end{align*}	
where $\zeta$ denotes the support of $\utilde{T}$. \par 
An estimator $d_{\Pi}$ that minimizes the Bayes risk $r(\Pi, d)$, among all estimators $d$ of $\mu_{Q}$, is called a Bayes estimator with respect to the prior $\Pi$. \par 
\noindent An estimator $d(\utilde{T})$ is said to be conditionally unbiased, if
$$ \mathbb{E}_{\underline{\theta}}((d(\utilde{T})|Q=q))=\mu_{q},~ \forall ~\underline{\theta} \in \Theta.$$ 
\noindent A naive estimator for estimating the mean effect of the selected treatment $\mu_Q$ is the weighted average of the sample means at the two stages, i.e. 
\begin{equation}
	d_M(\utilde{T})=\frac{n_1 \overline{X}_Q+n_2\overline{Y}}{n_1+n_2}.
\end{equation}
Clearly, $d_M(\utilde{T})$ is the maximum likelihood estimator (MLE) of $\mu_{Q}$. \par 
In the following subsection, the UMVCUE of the selected treatment mean $\mu_Q$ is derived under the assumption that the common variance is unknown.

 \subsection{The extended UMVCUE}
 \begin{theorem}
 	The two-stage UMVCUE of $\mu_{Q}$, given $Q$, is 
 	\begin{align}
 		d_U(\utilde{T})&=\frac{Z}{n_1+n_2}-\sqrt{\frac{n_1}{n_2(n_1+n_2)}}\widetilde{S} \frac{(1-{V^*}^2)^c}{2^{2c}c B(c,c)I_{c,c}\left(\frac{V^*+1}{2}\right)},
 	\end{align}
 	
 where, $Z=n_1\overline{X}_Q+n_2\overline{Y}$, $V=\frac{\sqrt{\frac{n_1(n_1+n_2)}{n_2}}}{\widetilde{S}} \left(\frac{Z}{n_1+n_2}-\overline{X}_{3-Q}\right)$, $V^*=\min\{V,1\}$, \vspace*{3mm}\\
 $\widetilde{S}^2=\sum\limits_{i=1}^{2} \sum\limits_{j=1}^{n_1} X_{ij}^2+\sum\limits_{j=1}^{n_2} Y_{j}^2-(n_1+n_2)\left(\frac{Z}{n_1+n_2}\right)^2-n_1 \overline{X}^2_{3-Q}$ , $c=\frac{2n_1+n_2-3}{2}$,\vspace*{3mm}\\
 $I_{c,c}(u)$ is the cumulative distribution function of a $Beta(c,c)$ distribution and $B(c,c)=\frac{\Gamma(c) \Gamma(c)}{\Gamma(2c)},~c >0$, is the usual beta function.
 \end{theorem}
\begin{proof} Let $\underline{S}=\left(\underline{X_1}, \underline{X_2},\underline{Y}\right)$ denote the vector of $n=2n_1+n_2$ sample observations of stage 1 and stage 2, combined. Then, the joint probability density function (p.d.f.) of $\underline{S}$, given $Q=q$, based on the first and second stage data can be written as
\begin{align*}
f_q(\underline{x_1},\underline{x_2},\underline{y}|\underline{\theta}) \propto \exp \left(\frac{1}{\sigma^2}\left[\sum\limits_{j=1}^{n_1}x_{qj}+\sum\limits_{j=1}^{n_2}y_{j}\right]\mu_q+ \frac{\sum\limits_{j=1}^{n_1}x_{(3-q) j}}{\sigma^2} \mu_ {3-q}-\frac{1}{2\sigma^2}\left[\sum\limits_{j=1}^{n_1}x^2_{qj}+\sum\limits_{j=1}^{n_1}x^2_{(3-q)j}+\sum\limits_{j=1}^{n_2}y^2_{j}\right]\right).
\end{align*}	
From the expression of the joint density $f_q(\underline{x_1},\underline{x_2},\underline{y}|\underline{\theta})$, we observe that, given $Q$, the statistic $\left(Z,\overline{X}_{3-Q}, \widetilde{S}^2\right)$ is a complete-sufficient statistic. Since, given $Q$, $\overline{Y}$ is an unbiased estimator of $\mu_{Q}$, the UMVCUE of $\mu_{Q}$ is the Rao-Blackwellization of $\overline{Y}$, conditional on the complete-sufficient statistic $\left(Z,\overline{X}_{3-Q}, \widetilde{S}^2\right)$. Therefore, the required UMVCUE of the selected mean $\mu_Q$, is $\mathbb{E}_{\underline{\theta}}\left(\overline{Y} \bigg|\left(Z,\overline{X}_{3-Q}, \widetilde{S}^2,Q\right)\right)$.\par
\noindent Define,
$$U=\frac{\sqrt{\frac{n_2(n_1+n_2)}{n_1}}\left(\overline{Y}-\frac{Z}{n_1+n_2}\right)}{\widetilde{S}}.$$
In order to show that the UMVCUE is the same as (2.4), it requires to show that
\begin{align}
\mathbb{E}_{\underline{\theta}}\left(U \bigg|\left(Z,\overline{X}_{3-Q}, \widetilde{S}^2\right)\right)&=-\frac{(1-{V^*}^2)^c}{2^{2c}c B(c,c)I_{c,c}\left(\frac{V^*+1}{2}\right)}.
\end{align}
To establish $(2.5)$, we need to obtain the conditional p.d.f. of $U$ given $\left(Z, \overline{X}_{3-Q}, \widetilde{S}^2, Q\right)$. \par \noindent Note that, on rearrangement of the terms in the expression of $\widetilde{S}^2$, we can write the expression of $U^2$ as
\begin{align}
U^2 &= \frac{\frac{n_2(n_1+n_2)}{n_1}\left(\overline{Y}-\frac{Z}{n_1+n_2}\right)^2}{\sum\limits_{i=1}^{2} \sum\limits_{j=1}^{n_1} (X_{ij}-\overline{X}_{i})^2+\sum\limits_{j=1}^{n_2} (Y_{j}-\overline{Y})^2+\frac{n_2(n_1+n_2)}{n_1}\left(\overline{Y}-\frac{Z}{n_1+n_2}\right)^2}.
\end{align}
As in \cite{cohen1989two} and \cite{robertson2019conditionally}, it can be verified that the conditional density of $U$, given $\left(Z, \overline{X}_{3-Q}, \widetilde{S}^2, Q\right)=\left(z,x_{3-q},\widetilde{s},q\right)$, is given by
\begin{align*}
	f_{U|\left(Z, \overline{X}_{3-Q}, \widetilde{S}^2, Q\right)}(u|(z,x_{3-q},\widetilde{s},q))&=\frac{(1-u^2)^{c-1}}{\displaystyle \int_{-1}^{v^*} (1-u^2)^{c-1} du},~~-1<u<v^*,
\end{align*}
where, for $v=\frac{\sqrt{\frac{n_1(n_1+n_2)}{n_2}}}{\widetilde{s}} \left(\frac{z}{n_1+n_2}-x_{3-Q}\right)$, $v^*=\min\{v,1\}$. \par 
Therefore, we have
\begin{align*}
	\mathbb{E}_{\underline{\theta}}\left(U \bigg|\left(Z,\overline{X}_{3-Q}, \widetilde{S}^2,Q\right)\right)&=\frac{\displaystyle \int_{-1}^{V^*} u (1-u^2)^{c-1}~du}{\displaystyle \int_{-1}^{V^*} (1-u^2)^{c-1}du}\\
	&=-\frac{(1-{V^*}^2)^c}{2^{2c}c B(c,c)I_{c,c}\left(\frac{V^*+1}{2}\right)}.
\end{align*}
Hence the result follows.
\end{proof}
Now, we will show that the naive estimator $d_M(\utilde{T})$, defined by $(2.3)$, is minimax and admissible for estimating $\mu_{Q}$ under the scaled mean squared error criterion.
\section{The minimax and admissible estimator}
Let ${S^*}^2=\frac{\sum\limits_{i=1}^{2} \sum\limits_{j=1}^{n_1} (X_{ij}-\overline{X}_{i})^2+\sum\limits_{j=1}^{n_2} (Y_{j}-\overline{Y})^2}{2n_1+n_2-3}$ be the pooled sample variance of the stage 1 and the stage 2 data, and let $W=(2n_1+n_2-3){S^*}^2$, so that $\frac{W}{\sigma^2} \sim \chi^2_{2n_1+n_2-3}$. Suppose that the unknown parameter vector $\underline{\theta}=(\mu_1,\mu_2,\sigma^2) \in \Theta$ is a realization of a random vector $\underline{P}=(P_1,P_2,P_3)$, having a specified probability distribution function. Consider a sequence of prior distributions (densities) $\left\{\xi_m\right\}_{m\geq 1}$, for  $\underline{P}=(P_1,P_2,P_3)$, such that :
\begin{itemize}
	\item[(i)] for any fixed $p_3 >0$, given $P_3 = p_3$, conditionally, $P_1$ and $P_2$ are independent and identically distributed as $N(0,m p_3)$;
	\item[(ii)] the random variable $P_3$ follows the inverse exponential distribution having the p.d.f.
\begin{equation*}
	\Pi_{1,P_3}(p_3)=\frac{1}{p_3^2} e^{-\frac{1}{p_3}}, ~p_3 >0. 
\end{equation*}
\end{itemize}
Recall that,  $\overline{X}_i ~(i=1,2)$ is the sample mean of the first stage sample from the $i^{th}$ population and $\overline{Y}$ is the second stage sample mean of the sample drawn from the population selected at the first stage. Then, under the prior distribution $\xi_m$, the joint posterior distribution of $ (P_1, P_2,P_3)$ given $(\overline{X}_1, \overline{X}_2,\overline{Y}, W)=(x_1,x_2,y,w)$, is such that:
\begin{itemize}
	\item[(i)] for any $p_3 >0$, conditionally, the random variables $ P_1$ and $P_2$ are independently distributed as $N\left(\theta_{1,m}, \tau^2_{1,m}\right)$ and $N\left(\theta_{2,m}, \tau^2_{2,m}\right)$, respectively, where 
	\begin{equation*}
		\left(\theta_{1,m}, \theta_{2,m}, \tau^2_{1,m}, \tau^2_{2,m}\right)=\begin{cases}
			\left(\frac{n_1x_1+n_2y}{n_1+n_2+\frac{1}{m}},\frac{n_1x_2}{n_1+\frac{1}{m}},\frac{1}{\frac{n_1+n_2}{p_3}+\frac{1}{m p_3}},\frac{1}{\frac{n_1}{p_3}+\frac{1}{m p_3}}\right), & \text{if $x_1  \geq x_2$ } \vspace{2mm}\\ 
			\left(\frac{n_1x_1}{n_1+\frac{1}{m}},\frac{n_1x_2+n_2y}{n_1+n_2+\frac{1}{m}}, \frac{1}{\frac{n_1}{p_3}+\frac{1}{m p_3}},\frac{1}{\frac{n_1+n_2}{p_3}+\frac{1}{m p_3}}\right), & \text{if $x_1 < x_2$}
		\end{cases}.
	\end{equation*}
\item [(ii)] the random variable $P_3$ follows the inverse gamma distribution having the p.d.f.
 $$\Pi_{2,m}(p_3)=\frac{(v_m)^{\frac{n+7}{2}}}{\Gamma(n+7)} \frac{e^{-\frac{v_m}{p_3}}}{p_3^{\frac{n+9}{2}}} $$
 where, $n=2 n_1+n_2-3$ and
 \begin{equation*}
 	v_m=\begin{cases}
 		1+\frac{w}{2}+\frac{n_1 x_1^2}{2}+\frac{n_1 x_2^2}{2}+\frac{n_2 y^2}{2}-\frac{(n_1 x_1+n_2y)^2}{2(n_1+n_2+\frac{1}{m})}-\frac{(n_1 x_2)^2}{2(n_1+\frac{1}{m})}, & \text{if $x_1  \geq x_2$ } \vspace{2mm}\\ 
 		1+\frac{w}{2}+\frac{n_1 x_1^2}{2}+\frac{n_1 x_2^2}{2}+\frac{n_2 y^2}{2}-\frac{(n_1 x_2+n_2y)^2}{2(n_1+n_2+\frac{1}{m})}-\frac{(n_1 x_1)^2}{2(n_1+\frac{1}{m})}, & \text{if $x_1 < x_2$}
 	\end{cases},
 \end{equation*}
$m=1,2,\ldots.$
\end{itemize}

Therefore, under the scaled squared error loss function \eqref{S2.E2}, the Bayes estimator of the selected treatment mean $\mu_Q$, w.r.t. the prior distribution $\xi_m$, is given by
\begin{align}
	d_{\xi_m}\left(\utilde{T}\right)
	&=\begin{cases*}
		\frac{\mathbb{E}^{\underline{P}|\utilde{T}}\left(\frac{P_1}{P_3}\right)}{\mathbb{E}^{\underline{P}|\utilde{T}}\left(\frac{1}{P_3}\right)} , & \text{if $\overline{X}_1 \geq \overline{X}_2$ }\\
		\frac{\mathbb{E}^{\underline{P}|\utilde{T}}\left(\frac{P_2}{P_3}\right)}{\mathbb{E}^{\underline{P}|\utilde{T}}\left(\frac{1}{P_3}\right)}, & \text{if $\overline{X}_1 <  \overline{X}_2$}
	\end{cases*} \nonumber\\
	&=\begin{cases*}
		\frac{n_1 \overline{X}_1 + n_2\overline{Y}}{n_1+n_2+\frac{1}{m}}, & \text{if $\overline{X}_1 \geq \overline{X}_2$ }\\
		\frac{n_1 \overline{X}_2 + n_2\overline{Y}}{n_1+n_2+\frac{1}{m}}, & \text{if $\overline{X}_1 <  \overline{X}_2$}
	\end{cases*} \nonumber\\
	&=\frac{n_1\overline{X}_Q+n_2\overline{Y}}{n_1+n_2+\frac{1}{m}},~m=1,2, \ldots.
\end{align}
The posterior risk of the Bayes rule $d_{\xi_m}\left(\utilde{T}\right)$ is obtained as
\begin{align*}
	r_{d_{\xi_m}}\left(\utilde{T}\right)&=\mathbb{E}^{\underline{P}|\utilde{T}} \left[\frac{\left(\frac{n_1\overline{X}_Q+n_2\overline{Y}}{n_1+n_2+\frac{1}{m}}-\mu_Q\right)^2}{P_3}\right]\\
	&=\frac{1}{n_1+n_2+\frac{1}{m}},
\end{align*}
which is independent of $\utilde{T}=(\overline{X}_1,\overline{X}_2,\overline{Y},W)$. Hence, the Bayes risk of the estimator $d_{\xi_m}$ is 
\begin{equation}
	r^*_{d_{\xi_m}}\left(\xi_m\right)=\frac{1}{n_1+n_2+\frac{1}{m}},~ m=1,2,\ldots
\end{equation}
Applying Lemma $2.2 ~(ii)$ of \cite{misra2022estimation}, we obtain the risk of the estimator $d_M(\utilde{T})=\frac{n_1\overline{X}_Q+n_2\overline{Y}}{n_1+n_2}$ as
\begin{align}
	R(\underline{\theta},d_{M})&=\mathbb{E}_{\underline{\theta}}\left( \frac{\frac{n_1\overline{X}_Q+n_2\overline{Y}}{n_1+n_2}-\mu_Q}{\sigma} \right)^2 \nonumber \\
	&=\frac{1}{n_1+n_2}.
\end{align}
Therefore, the Bayes risk of the naive estimator $d_{M}(\utilde{T})$, under the prior $\xi_m$, is
\begin{align}
	r^*_{d_{M}}\left(\xi_m\right)&=\frac{1}{n_1+n_2},~ m=1,2,\ldots
\end{align}
Now, we provide the following theorem which proves the minimaxity of the natural estimator $d_{M}$.
\begin{theorem}
	Under the scaled squared error loss function \eqref{S2.E2}, the natural estimator  $d_M(\utilde{T})=\frac{n_1\overline{X}_Q+n_2\overline{Y}}{n_1+n_2}$ is minimax for estimating the selected treatment mean $\mu_{Q}$.
\end{theorem}
\begin{proof}
	Let $d$ be any other estimator. Since, $d_{\xi_m}$, given by $(3.1)$, is the Bayes estimator of $\mu_{Q}$ under the prior $\xi_m, m=1,2,\ldots$, we have,
	\begin{align*}
		\sup_{\underline{\theta} \in \Theta } R(\underline{\theta}, d) &\geq \int_{\Theta} R(\underline{\theta},d) \xi_{m}(\underline{\theta})d\underline{\theta}\\
		&\geq \int_{\Theta} R(\underline{\theta},d_{\xi_{m}}) \xi_{m}(\underline{\theta})d\underline{\theta}\\
		&=r^*_{d_{\xi_{m}}}(\xi_{m})=\frac{1}{n_1+n_2+\frac{1}{m}}, ~~~m=1,2,\ldots ~~~(\text{using}~ (3.2))\\
		\Rightarrow \sup_{\underline{\mu} \in \Theta } R(\underline{\mu}, d) &\geq \lim_{m\to\infty}r^*_{d_{\xi_{m}}}(\xi_{m})=\frac{1}{n_1+n_2}=\sup_{\underline{\theta} \in \Theta } R(\underline{\theta}, d_{M}), ~~~(\text{using} ~(3.4))
	\end{align*}
	implying that $d_{M}(\utilde{T})=\frac{n_1\overline{X}_Q+n_2\overline{Y}}{n_1+n_2}$ is minimax for estimating $\mu_{Q}$.	
\end{proof}

We will now invoke the principle of invariance. The problem of estimating the selected treatment mean $\mu_{Q}$, under the scaled squared error loss function \eqref{S2.E2},  is invariant under the affine group of transformations and also under the group of permutations. 

It is easy to verify that any affine and permutation equivariant estimator of $\mu_{Q}$ will be of the form
\begin{equation} \label{S2,E3}
	d_{\psi}(\utilde{T})=\frac{n_1\overline{X}_Q+n_2\overline{Y}}{n_1+n_2} - \widetilde{S}\psi\left(\frac{D}{\widetilde{S}}\right),
\end{equation}
for some function $\psi:\mathbb{R} \rightarrow \mathbb{R}$, where $D=\frac{n_1\overline{X}_Q+n_2\overline{Y}}{n_1+n_2}-\overline{X}_{3-Q}$.  
Let the class of all affine and permutation equivariant estimators of the type (3.5) be denoted by $\mathscr{E}_1$. Clearly, the MLE $d_M(\utilde{T})$  and the UMVCUE $d_{U}$, defined by (2.3) and $(2.4)$  respectively, belong to the class $\mathscr{E}_1.$  

Next, we will show that the MLE $d_{M}$ is admissible within the class $\mathscr{E}_1$ of affine and permutation equivariant estimators. Notice that, the risk function of any estimator $d \in \mathscr{E}_1$ depends on $\underline{\theta} \in \Theta$ through $\mu=\frac{\theta_2-\theta_1}{\sigma}$, where $\theta_2=\max\left\{\mu_1,\mu_2\right\}$ and $\theta_1=\min\left\{\mu_1,\mu_2\right\}$. Consequently, the Bayes risk of any estimator $d \in \mathscr{E}_1$ depends on the prior distribution of $\underline{P}=(P_1, P_2, P_3)$ through the distribution of $\frac{P_{[2]} - P_{[1]} }{P_{3}}$, where $ P_{[1]}=\min \left\{P_1, P_2 \right\}$, $P_{[2]}=\max\left\{P_1, P_2\right\}.$ \par 
The following theorem establishes the admissibility of the estimator $d_{M}$, within the class $\mathscr{E}_1$ of affine and permutation equivariant estimators. Since the proof of the theorem is on the lines of the proof of Theorem 5.2 of \cite{masihuddin2021equivariant}, it is being omitted.
\begin{theorem}
	The natural estimator  $d_M(\utilde{T})=\frac{n_1\overline{X}_Q+n_2\overline{Y}}{n_1+n_2}$ is admissible for estimating the selected treatment mean $\mu_{Q}$, within the class $\mathscr{E}_1$, under the criterion of scaled mean squared error.
\end{theorem}

\section{Some additional naive estimators and a simulation study}
For estimating the selected treatment mean, some additional naive estimators can be obtained by plugging in the estimate of unknown variance $\sigma^2$ in the expressions of estimators ($\delta_{0}^{RB}$ and $\delta_1$) derived by \cite{misra2022estimation}, for the case when $\sigma^2$ is known. For unknown variance case, two such naive estimators are obtained as:

\begin{align}
d_{U1}(\utilde{T})=	d^{RB,\widehat{\sigma}}_{0}(\utilde{T})&=Z_1+ \frac{\sqrt{n_2}S^*}{\sqrt{n_1(n_1+n_2)}}\frac{\phi\left(\frac{\sqrt{n_1(n_1+n_2)}}{\sqrt{n_2}S^*}(Z_1-Z_2)\right)}{\Phi\left(\frac{\sqrt{n_1(n_1+n_2)}}{\sqrt{n_2}S^*}(Z_1-Z_2)\right)}
\end{align}
and 
\begin{align}
d_{U2}(\utilde{T})=	d^{\widehat{\sigma}}_{1}(\utilde{T})&=\begin{cases}
		\frac{(n_1+n_2)Z_1+n_1 Z_2}{2n_1+n_2}\left\{\frac{\Phi\left(\frac{Z_1-Z_2}{\widehat{\sigma_1}}\right)-\Phi\left(\frac{n_1(Z_1-Z_2)}{(2n_1+n_2)\widehat{\sigma_1}}\right)}{\Phi\left(\frac{Z_1-Z_2}{\widehat{\sigma_1}}\right)}\right\}\\+\frac{\widehat{\sigma_1}\phi\left(\frac{n_1(Z_1-Z_2)}{(2n_1+n_2)\widehat{\sigma_1}}\right)+Z_1\Phi\left(\frac{n_1(Z_1-Z_2)}{(2n_1+n_2)\widehat{\sigma_1}}\right)}{\Phi\left(\frac{Z_1-Z_2}{\widehat{\sigma_1}}\right)}, & \text{ if }~ Z_1 > Z_2\vspace{3mm}\\
		\frac{(n_1+n_2)Z_1+n_1 Z_2}{2n_1+n_2} , & \text{ if }~ Z_1 \leq  Z_2
	\end{cases} ,
\end{align}
where, $\widehat{\sigma_1}=S^*\sqrt{\frac{n_2}{n_1(n_1+n_2)}}$, $Z_1=\frac{n_1\overline{X}_Q+n_2\overline{Y}}{n_1+n_2}, Z_2=\overline{X}_{3-Q}$ and ${S^*}^2$ is the pooled sample variance based on the stage $1$ and the stage $2$ data. \par  Now we report a simulation study on performance comparison of estimators $d_{U_1}$, $d_{U_2}$, the MLE $d_{M}$ and the UMVCUE $d_{U}$ under the scaled mean squared error and the bias criterion. \par 
\vspace{4mm}

We compare the risk (scaled MSE) and the bias performances of various estimators of the selected treatment mean ($\mu_{Q}$) using the Monte-Carlo simulations. Following estimators are considered for our numerical study: $d_{M}$, $d_{U}$, $d_{U_1}$ and $d_{U_2}$ (see (2.3), (2.4), (4.1) and (4.2)). Since the scaled MSEs and biases of these estimators of $\mu_{Q}$ depend on parameters $\underline{\theta}=(\mu_{1},\mu_{2},\sigma) \in \Theta$  through $\mu~(\geq 0)$, we have plotted the simulated risks and the scaled biases of various estimators against $\mu \geq 0$ for different configurations of sample sizes $n_1$ and $n_2$. The simulated values of the scaled MSE and the bias based on 100,000 simulations are plotted in Figures 4.1-4.7. In Figures 4.1-4.2, we have plotted the simulated scaled MSEs of the proposed estimators against $\mu$. In Figures 4.3-4.5, we have plotted the simulated values of the scaled MSE against the information fraction~($=\frac{n_1}{n_1+n_2}$) for fixed total sample size $n=n_1+n_2$ and $\mu$. The scaled biases of various estimators have been plotted in Figures 4.6-4.7.\par
Following conclusions are drawn based on the simulation study :
\begin{itemize}
	\item[(i)] In conformity with (3.3), the scaled MSEs of the natural estimator $d_M$ is constant(=$\frac{1}{n_1+n_2}$) for all values of the normalized treatment effect difference $\mu$.
	
	\item[(ii)] The scaled MSEs of all other estimators ($d_U$, $d_{U1}$ and $d_{U2}$), except $d_M$, decrease as the the value of $\mu$ increases. For larger values of $\mu$, the estimator $d_{U2}$ has better scaled MSE performance as compared to other estimators.

	\item[(iii)] For $n_1=n_2$, the UMVCUE $d_U$ has the similar scaled MSE performance as the estimator $d_{U1}$.  For $n_1 \geq n_2$, the estimator $d_{U2}$ outperforms estimators $d_U$ and $d_{U1}$. \par 
	It is also observed that the estimator $d_{U2}$ uniformly dominates $d_{U1}$ in terms of the scaled MSE.\par 
	
	\item [(iv)] For small $n_2$ and relatively large $n_1$, the scaled MSEs of the estimators $d_{U1}$ and $d_{U2}$ are very much close to that of $d_M$.\par 
	When $n_1$ is small and $n_2$ is large, the scaled MSE of the UMVCUE $d_{U}$ is closer to the scaled MSEs of $d_M$.
	
	\item[(v)] For fixed smaller values of  $\mu$ and the total sample size $n=n_1+n_2$, the scaled MSEs of the UMVCUE $d_U$ increases as the information fraction $\frac{n_1}{n}$ increase whereas the scaled MSEs of $d_{U1}$ and $d_{U2}$ decreases as the information fraction increases. For larger values of $\mu$, the estimators $d_U$ and $d_M$ have similar scaled MSE performance.
	
     \item [(vi)] The scaled biases of all the competing estimators ($d_M$, $d_{U1}$ and $d_{U2}$), except the UMVCUE $d_U$, decreases as the value of $\mu$ increases. \par 
	 For larger values of $\mu$ these estimators have scaled bias performance comparable to $d_U$. \par It is also interesting to note that all the competing estimators of $\mu_{Q}$ have the maximum bias when $\mu$ is zero and it decreases as the value of $\mu$ increases.
	 
	 \item[(vii)] The scaled MSEs and scaled biases of all the competing estimators decrease towards zero for large value of the sample sizes $n_1$ and $n_2$.

\end{itemize}
\vspace{2mm}
 When scaled MSE is the key criterion for choosing suitable estimators, we recommend estimators $d_M$ and $d_{U2}$. For some specific configurations of $n_1$ and $n_2$ (small $n_1$ and large $n_2$), the UMVCUE $d_U$ is also a good competitor. \par 
 \noindent When both scaled bias and scaled MSE are to be controlled, we recommend using estimators $d_U$ and $d_M$.
\newpage
\FloatBarrier
\begin{figure}
	\centering
	\begin{subfigure}[b]{0.4\textwidth}
		\centering
		\includegraphics[width=9cm,height=8cm]{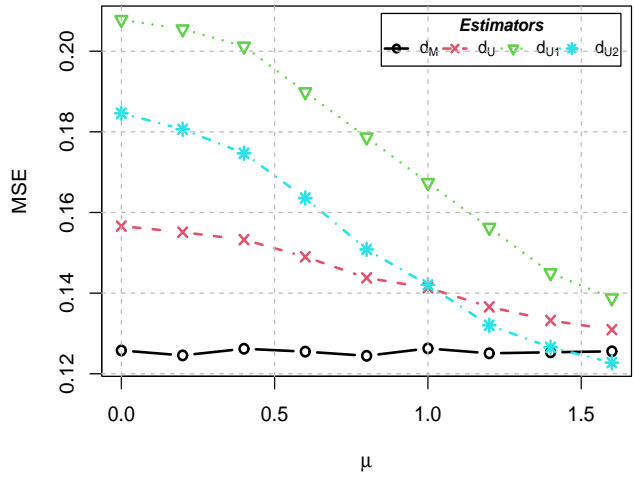}
		\caption{$n_1=3, n_2=5$}
	\end{subfigure}
	\hfill
	\begin{subfigure}[b]{0.4\textwidth}
		\centering
		\includegraphics[width=9cm,height=8cm]{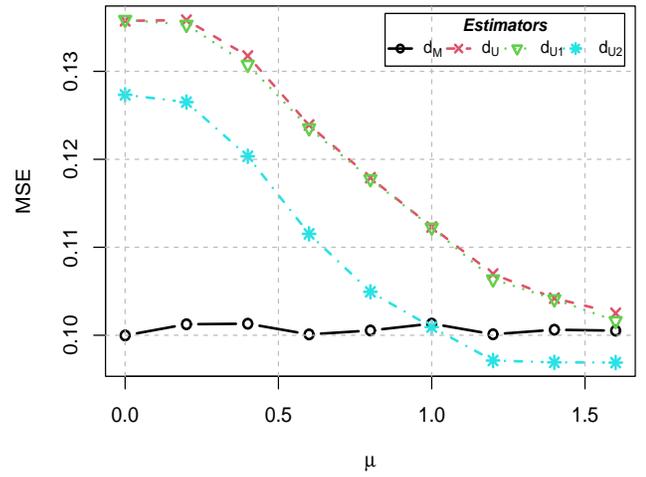}
			\caption{$n_1=5, n_2=5$}
	\end{subfigure}
	\begin{subfigure}[b]{0.4\textwidth}
		\centering
		\includegraphics[width=9cm,height=8cm]{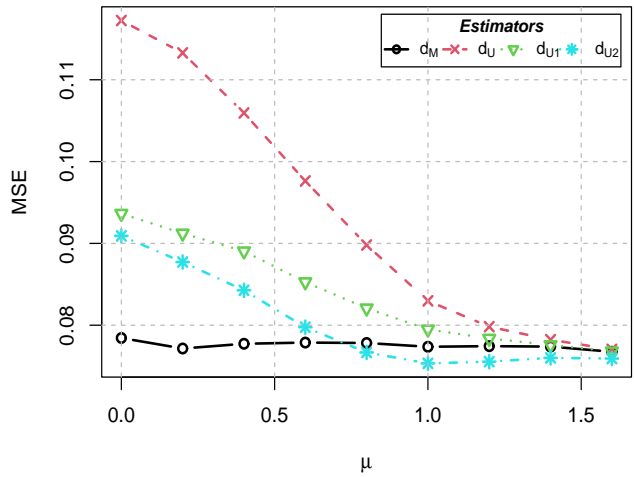}
		\caption{$n_1=8, n_2=5$}
	\end{subfigure}
	\hfill
	\begin{subfigure}[b]{0.4\textwidth}
		\centering
		\includegraphics[width=9cm,height=8cm]{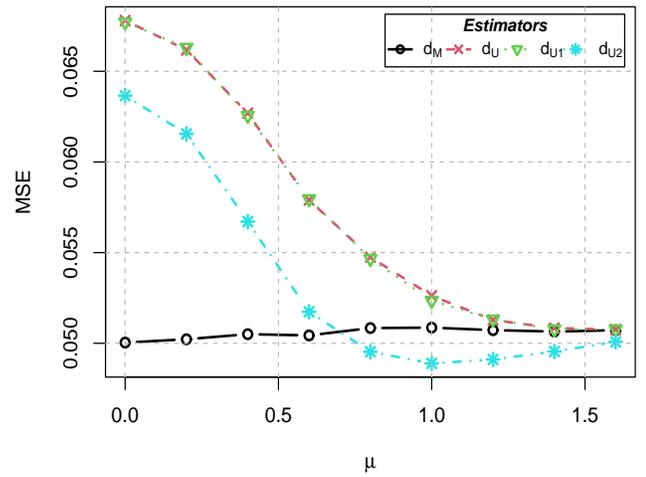}
		\caption{$n_1=10, n_2=10$}
	\end{subfigure}
	\caption{\textbf{Scaled MSE plots of various competing estimators for different configurations of $n_1$ and $n_2$}}
\end{figure}
 \FloatBarrier
 
 \FloatBarrier
 \begin{figure}
 	\centering
 	\begin{subfigure}[b]{0.4\textwidth}
 		\centering
 		\includegraphics[width=9cm,height=8cm]{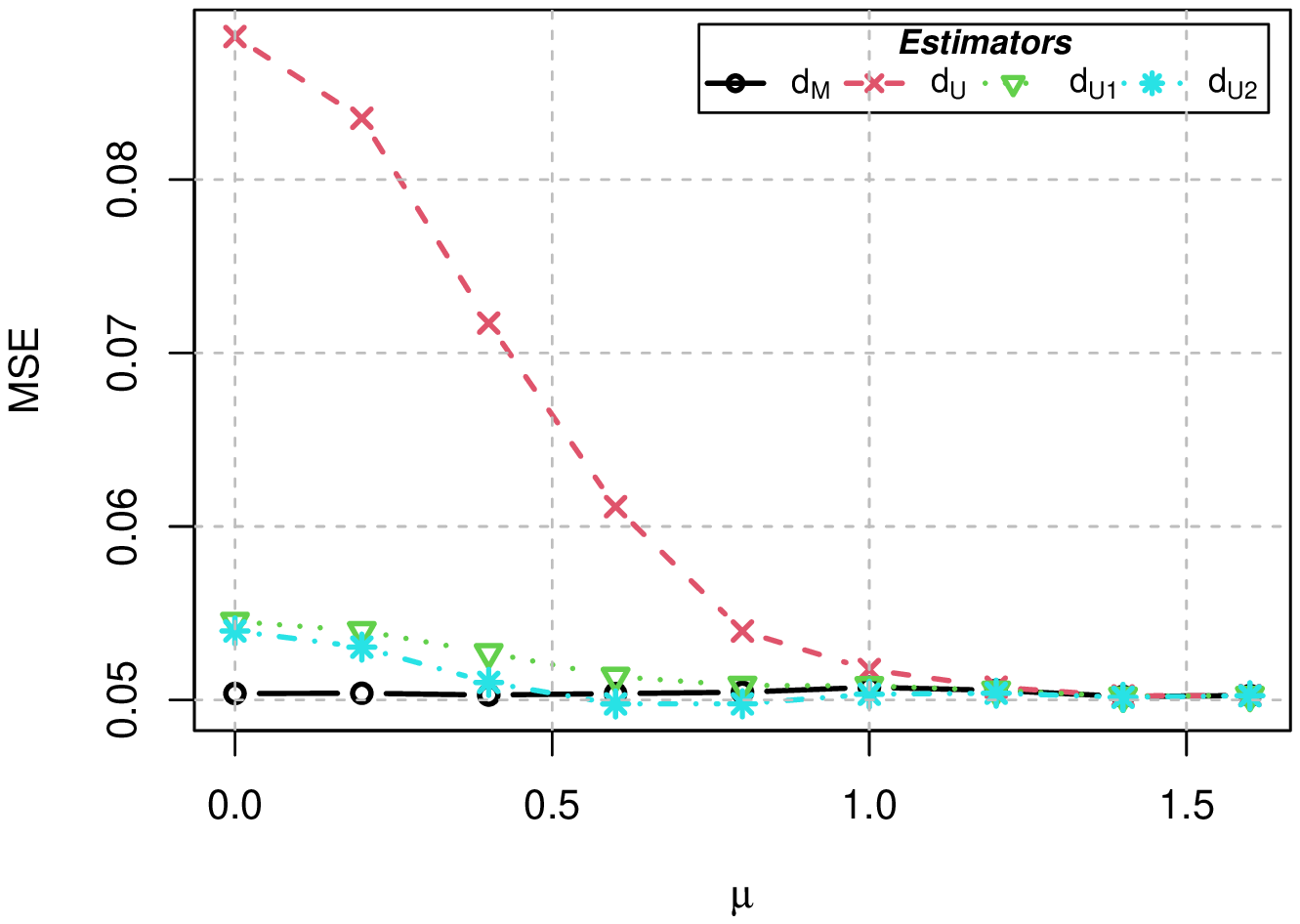}
 		\caption{$n_1=15, n_2=5$}
 	\end{subfigure}
 	\hfill
 	\begin{subfigure}[b]{0.4\textwidth}
 		\centering
 		\includegraphics[width=9cm,height=8cm]{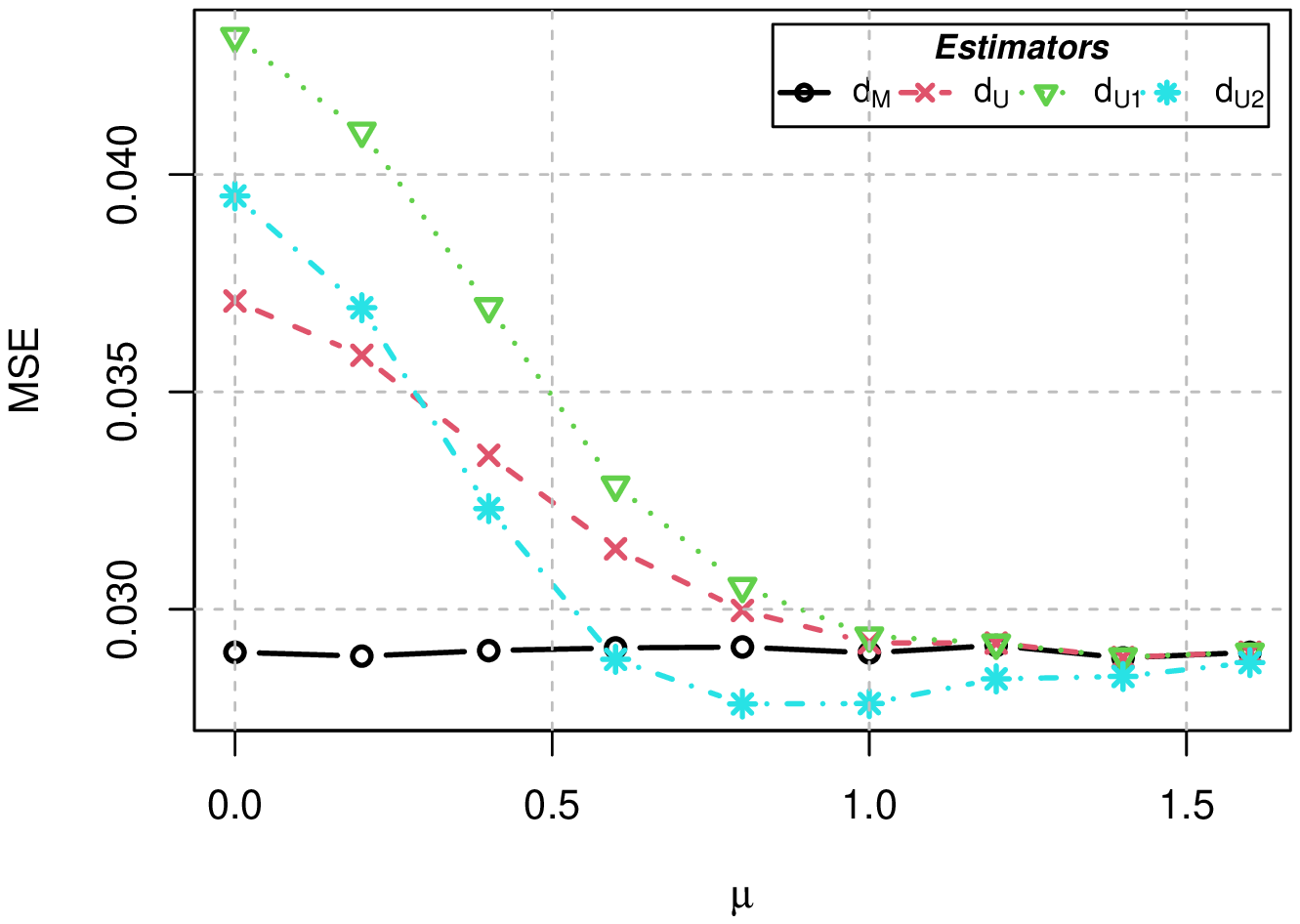}
 		\caption{$n_1=15, n_2=20$}
 	\end{subfigure}
 	\begin{subfigure}[b]{0.4\textwidth}
 		\centering
 		\includegraphics[width=9cm,height=8cm]{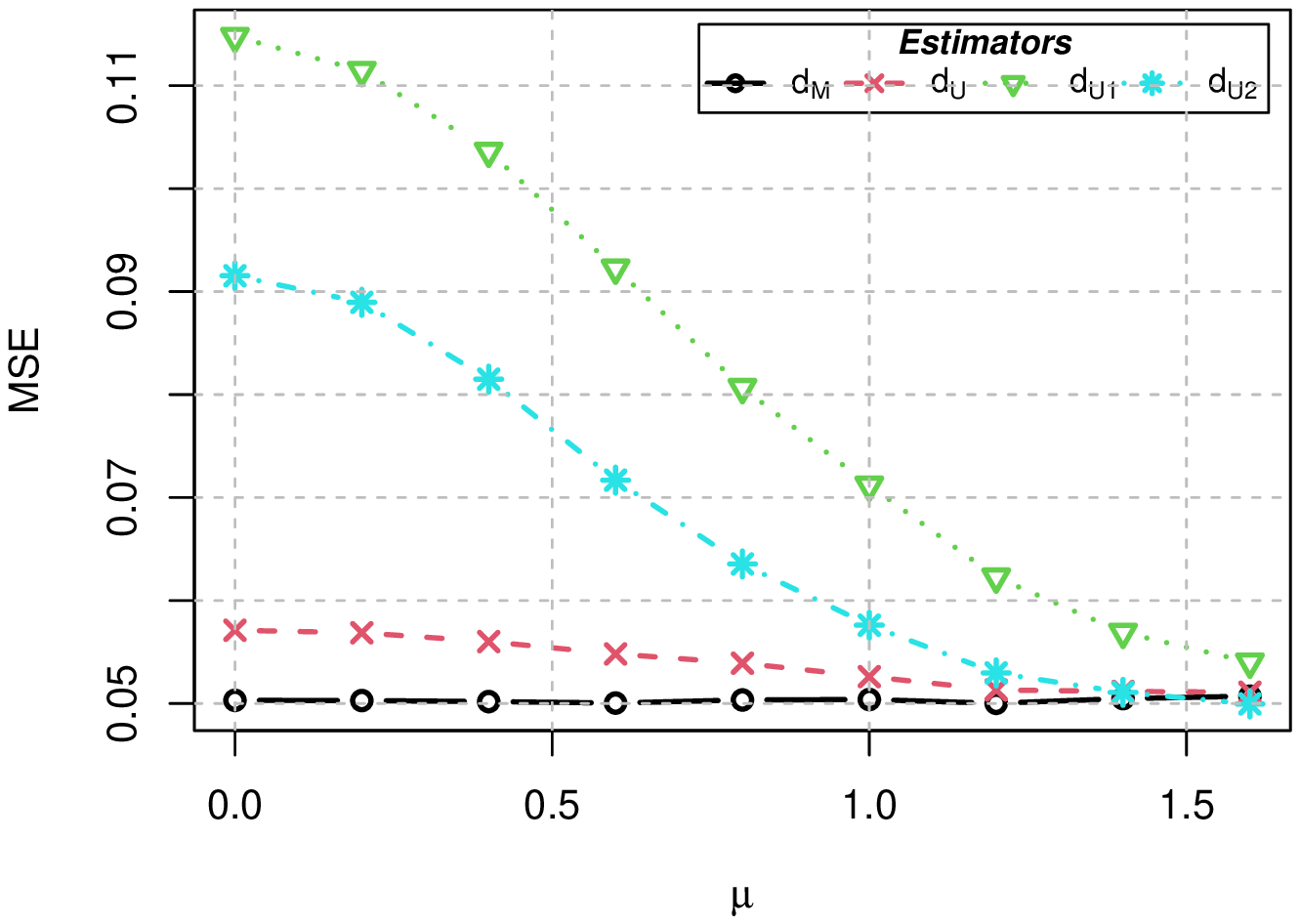}
 		\caption{$n_1=5, n_2=15$}
 	\end{subfigure}
 	\hfill
 	\begin{subfigure}[b]{0.4\textwidth}
 		\centering
 		\includegraphics[width=9cm,height=8cm]{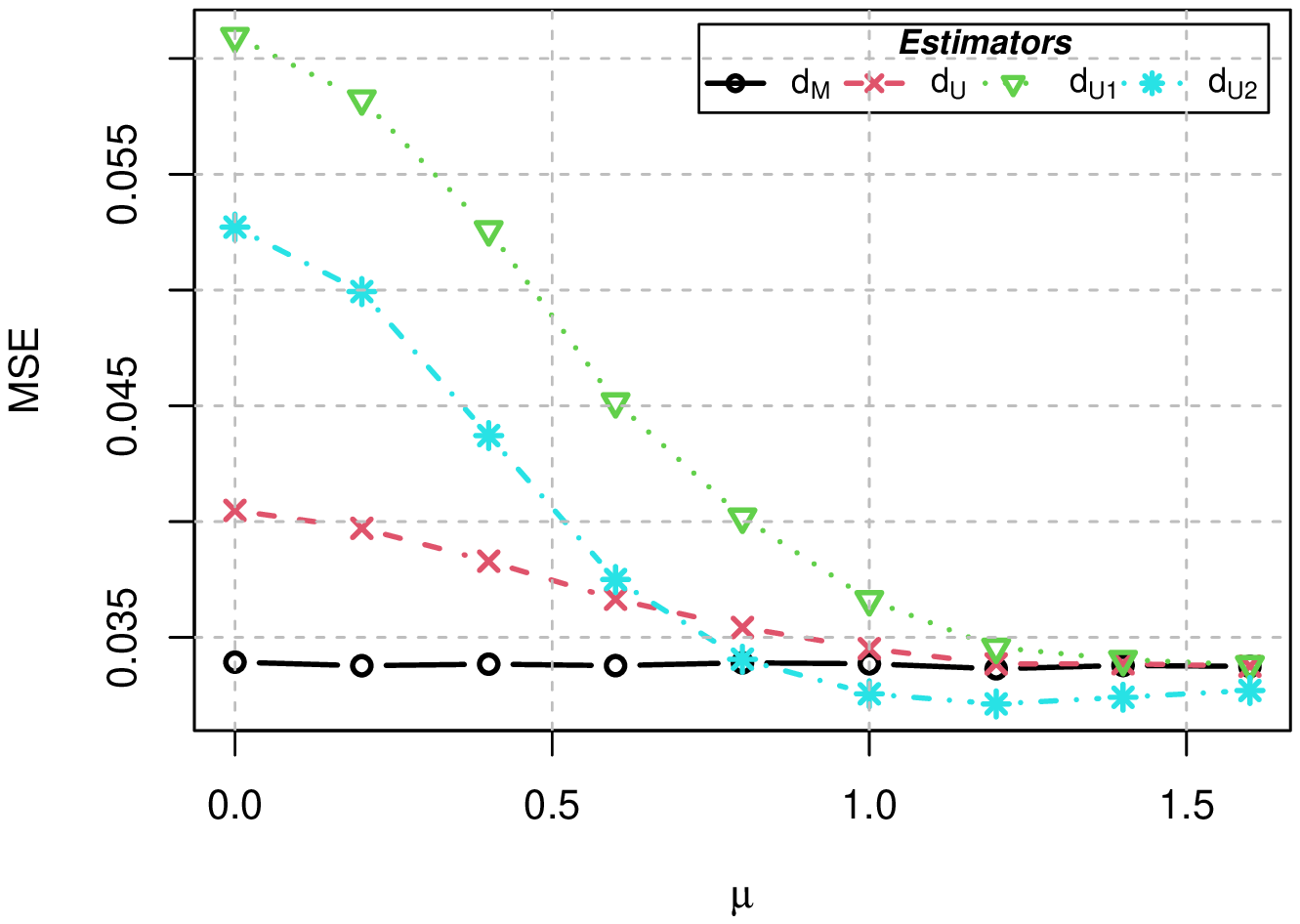}
 		\caption{$n_1=10, n_2=30$}
 	\end{subfigure}
 	\caption{\textbf{Scaled MSE plots of various competing estimators for different configurations of $n_1$ and $n_2$}}
 \end{figure}
 \FloatBarrier

\FloatBarrier
\begin{figure}[!h]
	\centering
	\includegraphics[width=4.3in,height=3.4in]{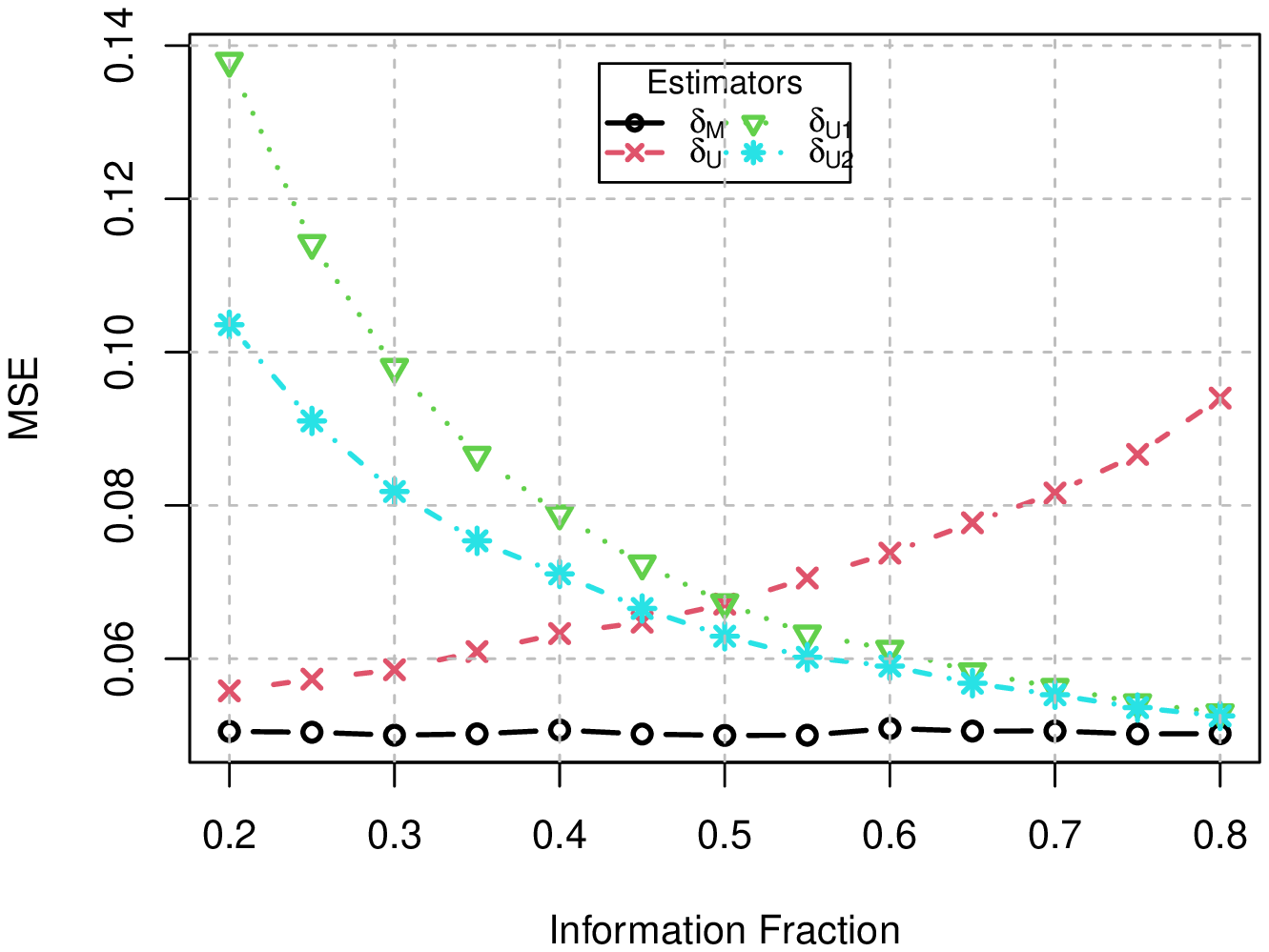}
	\caption{\textbf{Scaled MSE plots of different competing estimators for fixed $n=n_1+n_2=20$ and $\mu=0.1$.}}
\end{figure}
\begin{figure}[!h]
	\centering
	\includegraphics[width=4.3in,height=3.4in]{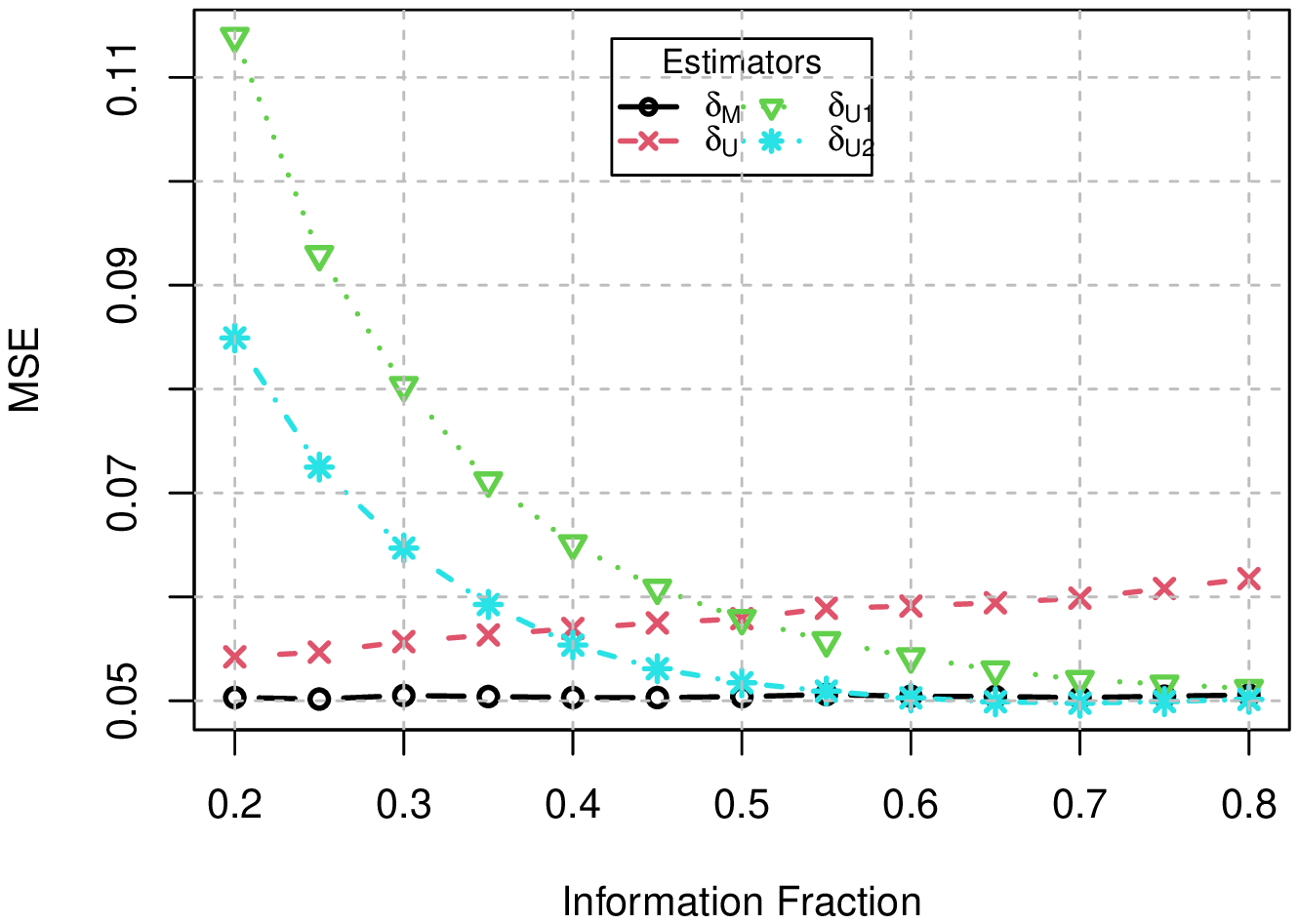}
	\caption{\textbf{Scaled MSE plots of different competing estimators for for fixed $n=n_1+n_2=20$ and $\mu=0.6$.}}
\end{figure}
\FloatBarrier
\FloatBarrier
\begin{figure}[!h]
	\centering
	\includegraphics[width=4.3in,height=3.4in]{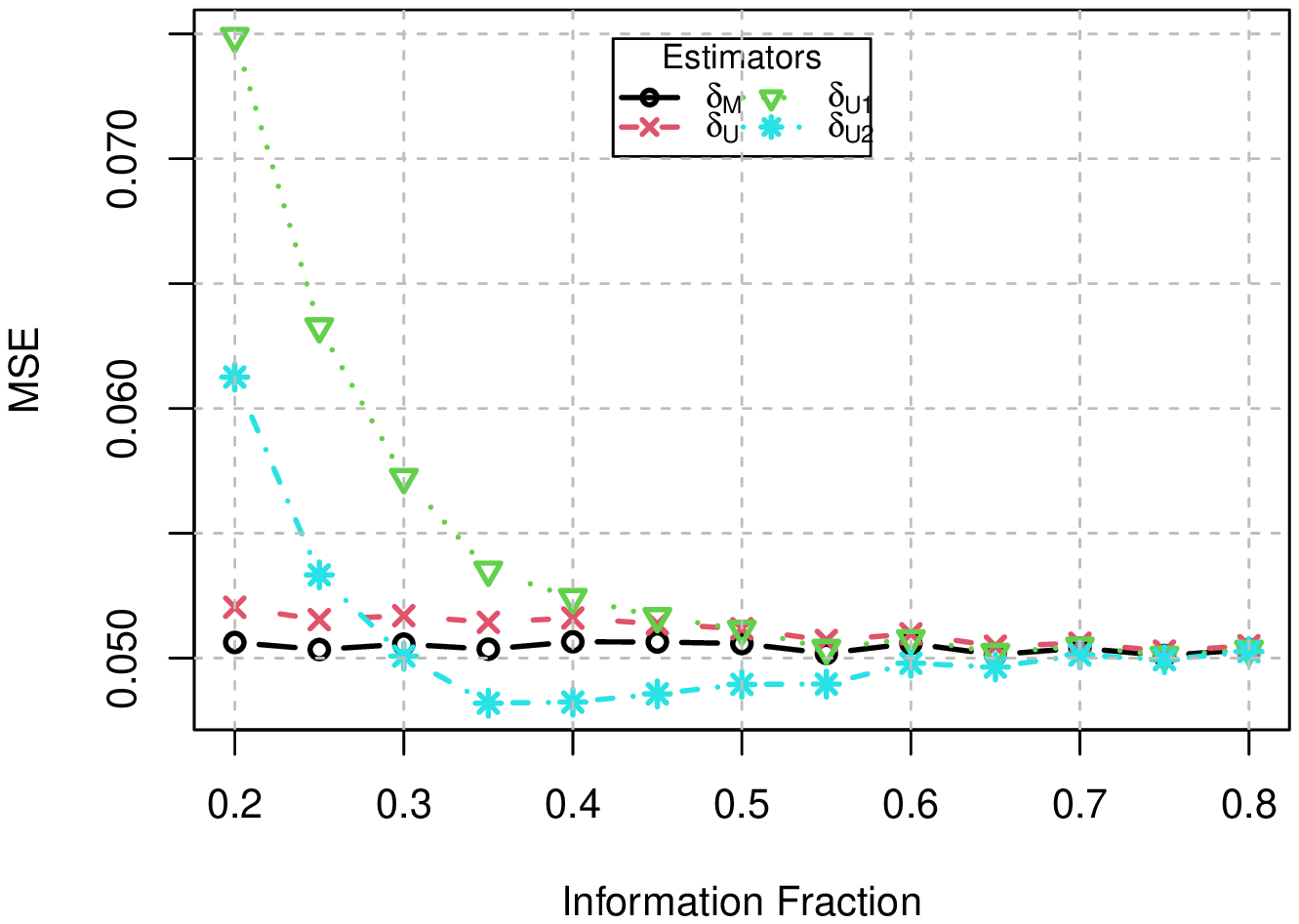}
	\caption{\textbf{Scaled MSE plots of different competing estimators for for fixed $n=n_1+n_2=20$ and $\mu=1.2$.}}
\end{figure}
\begin{figure}[!h]
	\centering
	\includegraphics[width=4.3in,height=3.4in]{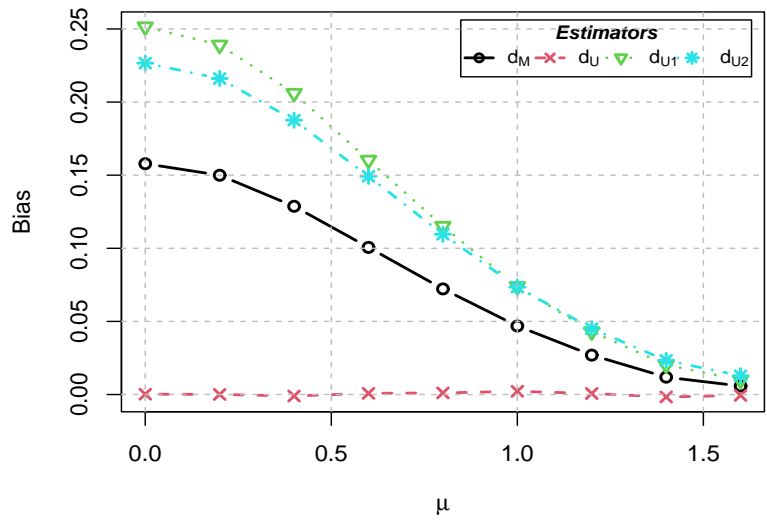}
	\caption{\textbf{Scaled Bias plots of different competing estimators for $n_1=5,n_2=3$.}}
\end{figure}
\FloatBarrier
\begin{figure}
	\centering
	\begin{subfigure}[b]{0.4\textwidth}
		\centering
		\includegraphics[width=9cm,height=8cm]{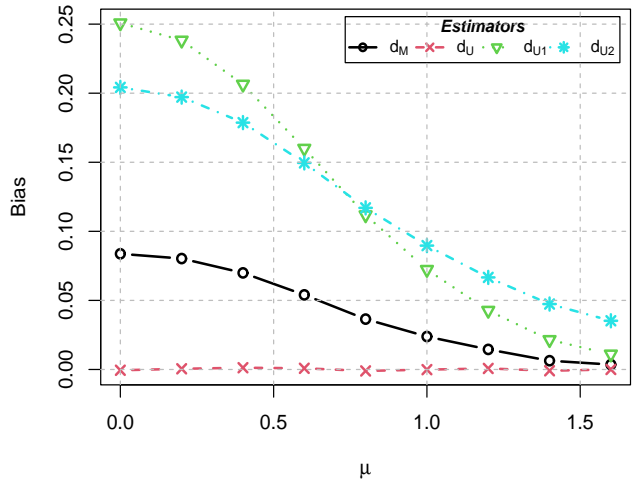}
		\caption{$n_1=5, n_2=5$}
	\end{subfigure}
	\hfill
	\begin{subfigure}[b]{0.4\textwidth}
		\centering
		\includegraphics[width=9cm,height=8cm]{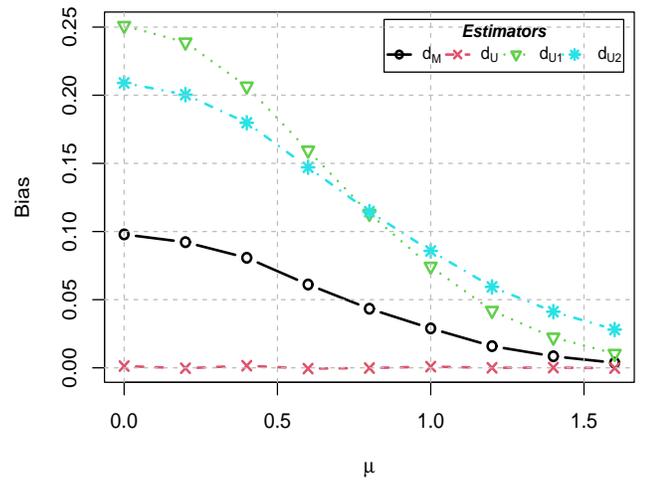}
		\caption{$n_1=5, n_2=8$}
	\end{subfigure}
	\begin{subfigure}[b]{0.4\textwidth}
		\centering
		\includegraphics[width=9cm,height=8cm]{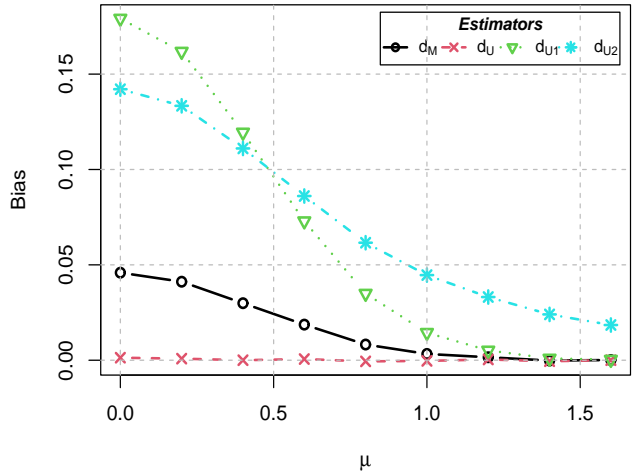}
		\caption{$n_1=10, n_2=30$}
	\end{subfigure}
	\hfill
	\begin{subfigure}[b]{0.4\textwidth}
		\centering
		\includegraphics[width=9cm,height=8cm]{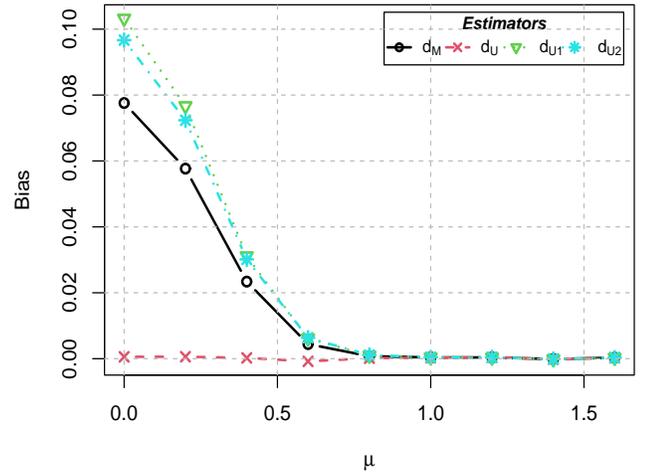}
		\caption{$n_1=30, n_2=10$}
	\end{subfigure}
	\caption{\textbf{Scaled Bias plots of various competing estimators for different configurations of $n_1$ and $n_2$}}
\end{figure}
\FloatBarrier
 \section{Real data example}
 In this section, we provide an illustration of the theoretical findings of our paper to a data set.  The details of the data set can be accessed using the link: \url{https://vincentarelbundock.github.io/Rdatasets/doc/Stat2Data/FatRats.html}. The data is presented in Table 1 below. Data from this experiment compared weight gain for 60 baby rats that were fed different diets. Half of the rats were given low-protein diets and the rest were supplied high-protein diet. The source of protein was either beef, cereal, or pork. \par 
 Using the Shapiro-Wilk normality test with a $p$- value of 0.834 for high protein diet and 0.771 for low protein diet, we conclude that the underlying populations of weight gains of the baby rats receiving high protein diet and low protein diet are approximately Gaussian. The assumption of equality of variance of the two populations is also accepted with a $p$- value of 0.631 by using the F test. So, the data can be considered to have come from two normal populations $N(\mu_{1},\sigma^2)$ and $N(\mu_{2},\sigma^2)$, with $\left(\widehat{\mu_{1}},\widehat{\mu_{2}}\right)=\left(92.5,81.6\right)$.\par 
 To select the effective protein diet (which provides the largest weight gain), we extract a sample of size $n_1=20$ from the each population group and calculate $\overline{X}_1$ and $\overline{X}_2$. If $\overline{x}_1 > \overline{x}_2 $, we select the population corresponding to the high protein diet and, if $\overline{x}_2 \geq \overline{x}_1 $, we select the population corresponding to the low protein diet. In stage 2, we draw an additional sample of size $n_2=10$ from the selected population in stage 1, and calculate the stage 2 sample mean $\overline{Y}$. Finally, we compute the estimates $d_{M}$, $d_{U}$, $d_{U1}$ and $d_{U2}$ based on the calculated values of $\overline{X}_1$, $\overline{X}_2$, ${S}^2$  and $\overline{Y}$ .

 \FloatBarrier
 \begin{table}[h!]
 	\centering
 	\textbf{Table 1}\\
 	\textbf{Weight gain in rats after being fed with high-protein diets.}\\\textbf{Stage I data}
 	\begin{tabular}{lllllllllll}
 		73 & 102 & 118 &  & 104 & 81 &  & 107 & 100 &  &   \\
 		87 & 117 & 111 &  & 98 & 74 &  &  56 & 111 &  &   \\
 		95 &  88 & 82 &  &  77&  86 &  &  92 &  &  &   \\
 		
 	\end{tabular}

 \end{table}
 \vskip -0.2in
\begin{table}[h!]
 	\centering
 	\textbf{Weight gain in rats after being fed with low-protein diets.}\\	\textbf{Stage I data}
 	\begin{tabular}{lllllllllll}
 		90& 76 & 90 & &64 &  86 & &51& 72& &   \\
 		90 & 95 & 78 & &107 & 107 & &97 & 80  & &   \\
 		 98 &  74 &  74 & &67 & 89 & &58 &  &  &   \\
 	
 	\end{tabular}
 \end{table}
 \begin{table}[h!]
 	\centering
	\textbf{Weight gain in rats after being fed with high-protein diets.}\\\textbf{Stage II data}
 	\begin{tabular}{lllllllllll}
		94 & 79 & 96 &  & 98 & 102 &  & 102 & 108 &  &   \\
 		91 & 120 & 105 &  &  &  &  &  & &  &   \\

 	\end{tabular}
 
 \end{table}
 \FloatBarrier
 The various two-stage estimates of the selected treatment mean $\mu_{Q}$ are tabulated in Table 2 below:
 \begin{center}
 	\textbf{Table 2:} \textbf{Various estimates of the selected treatment mean $\mu_{Q}$}.
 	\label{tab:table1}\\
 	\begin{tabular}{|l|c|c|c|} \hline
 		\textbf{$d_{M}$} &  \textbf{$d_{U}$} &\textbf{$d_{U1}$} & \textbf{$d_{U2}$}   \\
 		
 		\hline
 		95.13 & 91.23& 88.34 & 89.34\\ \hline
 		
 	\end{tabular}
 \end{center}
From various estimates of the selected treatment mean given in Table 2, we notice that the two-stage adaptive estimates $d_M$ and $d_U$
are near the true value of the mean of the selected treatment,
i.e., $\mu_Q$ = 92.5. 

\section{Concluding Remarks}
In this article, we investigated the problem of estimating the mean of the selected treatment under a two-stage adaptive design. We have extended the work of \cite{cohen1989two} by considering availability of multiple observations at the second stage.  \par 
In the unknown variance setting, we have derived the UMVCUE of the selected treatment mean and have shown that the MLE (a weighted
average of the first and second stage sample means, with weights being proportional
to the corresponding sample sizes) is minimax and admissible. We have also proposed two plug-in estimators ($d_{U1}$ and $d_{U2}$) obtained by plugging in the pooled sample variance in place of the common variance $\sigma^2$, in some of the estimators proposed by \cite{misra2022estimation}.
Using simulations we have tried to investigate the strengths and weaknesses of these four estimators under different trial designs.

\bibliographystyle{apalike}
\bibliography{ref3}

\begin{thebibliography}{}

\bibitem[Bahadur and Goodman, 1952]{bahadur1952impartial}
Bahadur, R.~R. and Goodman, L.~A. (1952).
\newblock Impartial decision rules and sufficient statistics.
\newblock {\em The Annals of Mathematical Statistics}, pages 553--562.

\bibitem[Bauer and Kieser, 1999]{bauer1999combining}
Bauer, P. and Kieser, M. (1999).
\newblock Combining different phases in the development of medical treatments
  within a single trial.
\newblock {\em Statistics in medicine}, 18(14):1833--1848.

\bibitem[Bowden and Glimm, 2008]{bowden2008unbiased}
Bowden, J. and Glimm, E. (2008).
\newblock Unbiased estimation of selected treatment means in two-stage trials.
\newblock {\em Biometrical Journal: Journal of Mathematical Methods in
  Biosciences}, 50(4):515--527.

\bibitem[Bowden and Glimm, 2014]{bowden2014conditionally}
Bowden, J. and Glimm, E. (2014).
\newblock Conditionally unbiased and near unbiased estimation of the selected
  treatment mean for multistage drop-the-losers trials.
\newblock {\em Biometrical Journal}, 56(2):332--349.

\bibitem[Carreras and Brannath, 2013]{carreras2013shrinkage}
Carreras, M. and Brannath, W. (2013).
\newblock Shrinkage estimation in two-stage adaptive designs with midtrial
  treatment selection.
\newblock {\em Statistics in Medicine}, 32(10):1677--1690.

\bibitem[Chiu et~al., 2018]{chiu2018design}
Chiu, Y.-D., Koenig, F., Posch, M., and Jaki, T. (2018).
\newblock Design and estimation in clinical trials with subpopulation
  selection.
\newblock {\em Statistics in medicine}, 37(29):4335--4352.

\bibitem[Cohen and Sackrowitz, 1989]{cohen1989two}
Cohen, A. and Sackrowitz, H.~B. (1989).
\newblock Two stage conditionally unbiased estimators of the selected mean.
\newblock {\em Statistics \& Probability Letters}, 8(3):273--278.

\bibitem[Eaton, 1967]{eaton1967some}
Eaton, M.~L. (1967).
\newblock Some optimum properties of ranking procedures.
\newblock {\em The Annals of Mathematical Statistics}, 38(1):124--137.

\bibitem[Kimani et~al., 2020]{kimani2020point}
Kimani, P.~K., Todd, S., Renfro, L.~A., Glimm, E., Khan, J.~N., Kairalla,
  J.~A., and Stallard, N. (2020).
\newblock Point and interval estimation in two-stage adaptive designs with time
  to event data and biomarker-driven subpopulation selection.
\newblock {\em Statistics in medicine}, 39(19):2568--2586.

\bibitem[Kimani et~al., 2013]{kimani2013conditionally}
Kimani, P.~K., Todd, S., and Stallard, N. (2013).
\newblock Conditionally unbiased estimation in phase ii/iii clinical trials
  with early stopping for futility.
\newblock {\em Statistics in Medicine}, 32(17):2893--2910.

\bibitem[Masihuddin and Misra, 2021]{masihuddin2021equivariant}
Masihuddin and Misra, N. (2021).
\newblock Equivariant estimation following selection from two normal
  populations having common unknown variance.
\newblock {\em Statistics}, 55(6):1407--1438.

\bibitem[Masihuddin and Misra, 2022]{misra2022estimation}
Masihuddin and Misra, N. (2022).
\newblock Estimation of the selected treatment mean in two-stage
  drop-the-losers design.
\newblock {\em arXiv preprint arXiv:2209.08567}.

\bibitem[Misra and Dhariyal, 1994]{misra1994non}
Misra, N. and Dhariyal, I.~D. (1994).
\newblock Non-minimaxity of natural decision rules under heteroscedasticity.
\newblock {\em Statistics and Decisions}, (12):79--98.

\bibitem[Misra and Singh, 1993]{misra1993umvue}
Misra, N. and Singh, G. (1993).
\newblock On the umvue for estimating the parameter of the selected exponential
  population.
\newblock {\em Journal of Indian Statistical Association}, 31(1):61--9.

\bibitem[Pallmann et~al., 2018]{pallmann2018adaptive}
Pallmann, P., Bedding, A.~W., Choodari-Oskooei, B., Dimairo, M., Flight, L.,
  Hampson, L.~V., Holmes, J., Mander, A.~P., Odondi, L., Sydes, M.~R., et~al.
  (2018).
\newblock Adaptive designs in clinical trials: why use them, and how to run and
  report them.
\newblock {\em BMC medicine}, 16(1):1--15.

\bibitem[Putter and Rubinstein, 1968]{putter1968estimating}
Putter, J. and Rubinstein, D. (1968).
\newblock On estimating the mean of a selected population technical report no.
  165.
\newblock {\em Department of Statistics, University of Wisconsin}.

\bibitem[Robertson et~al., 2022]{robertson2022point}
Robertson, D.~S., Choodari-Oskooei, B., Dimairo, M., Flight, L., Pallmann, P.,
  and Jaki, T. (2022).
\newblock Point estimation for adaptive trial designs i: A methodological
  review.
\newblock {\em Statistics in Medicine}.

\bibitem[Robertson and Glimm, 2019]{robertson2019conditionally}
Robertson, D.~S. and Glimm, E. (2019).
\newblock Conditionally unbiased estimation in the normal setting with unknown
  variances.
\newblock {\em Communications in Statistics-Theory and Methods},
  48(3):616--627.

\bibitem[Sampson and Sill, 2005]{sampson2005drop}
Sampson, A.~R. and Sill, M.~W. (2005).
\newblock Drop-the-losers design: normal case.
\newblock {\em Biometrical Journal: Journal of Mathematical Methods in
  Biosciences}, 47(3):257--268.

\bibitem[Stallard and Friede, 2008]{stallard2008group}
Stallard, N. and Friede, T. (2008).
\newblock A group-sequential design for clinical trials with treatment
  selection.
\newblock {\em Statistics in Medicine}, 27(29):6209--6227.

\bibitem[Tappin, 1992]{tappin1992unbiased}
Tappin, L. (1992).
\newblock Unbiased estimation of the parameter of a selected binomial
  population.
\newblock {\em Communications in Statistics-Theory and Methods},
  21(4):1067--1083.

\end{thebibliography}

\end{document}